\documentclass[11pt]{amsart}

\usepackage[utf8]{inputenc}
\usepackage{amsthm, amsmath, amssymb, tikz, bm}
\usepackage{verbatim}
\usepackage{graphicx}
\usepackage[linesnumbered,ruled,vlined]{algorithm2e}
\usepackage{float}
\usepackage{mathtools}
\usepackage{CJK} 
\usepackage{xifthen}
\usepackage[colorlinks=true,
linkcolor=blue,citecolor=blue,
urlcolor=blue]{hyperref}

\usepackage[margin=1.2in]{geometry}
\usepackage{extarrows}

\usepackage[shortlabels]{enumitem}
\usepackage{todonotes}
\usetikzlibrary{calc,shapes, backgrounds}
\allowdisplaybreaks

\newtheorem{lemma}{Lemma}
\newtheorem{theorem}{Theorem}
\newtheorem{corollary}{Corollary}
\newtheorem{definition}{Definition}
\newtheorem{example}{Example}

\newtheorem{remark}{Remark}

\usepackage{xcolor}

\newcommand{\ds}{\displaystyle}
\newcommand{\dss}{\displaystyle\sum}

\newcommand{\lp}{\left(}
\newcommand{\rp}{\right)}

\DeclarePairedDelimiter{\abs}{\lvert}{\rvert}%

\title{Ollivier Ricci-flow on weighted graphs}

\author{
Shuliang Bai 
\and 
Yong Lin
\and
Linyuan Lu
\and 
Zhiyu Wang 
\and 
Shing-Tung Yau}
\thanks{Shuliang Bai is supported in part by NSFC grant number 12301434. Yong Lin is supported in part by NSFC grant number 12071245.
Linyuan Lu is supported in part by NSF grant DMS 2038080.}

\date{}

\begin{document}
\maketitle
\begin{abstract}
We study the existence of solutions of Ricci flow equations of Ollivier-Lin-Lu-Yau curvature defined on weighted graphs. Our work is 
 motivated by\cite{NLLG} in which the
 discrete time Ricci flow algorithm has been applied successfully as a discrete geometric approach in detecting communities.  Our main result is the existence and uniqueness theorem for solutions to a continuous time normalized Ricci flow. We also display  possible solutions to the Ricci flow on path graph and prove the Ricci flow on finite star graph with at least three leaves converges to constant-weighted star. 

\end{abstract}


\section{Introduction}

Ricci flow which was introduced by Richard Hamilton \cite{hami1982} in 1980s is a process that deforms the metric of a Riemannian manifold in a way formally analogous to the diffusion of heat.    
On Riemannian manifolds $M$ with a smooth Riemannian metric  $g$,  the geometry of $(M, g)$ is altered by changing the metric $g$ via a second-order nonlinear PDE on symmetric $(0, 2)$-tensors: \begin{align}\label{equ:manifoldRicciflow}
\frac{d}{d t}g_{ij} = -2R_{ij},
\end{align}
where $R_{ij}$ is the Ricci curvature.
A solution to a Ricci flow is a one-parameter family of metrics $g(t)$ on a smooth manifold $M$, defined on a time interval $I$, and satisfying equation (\ref{equ:manifoldRicciflow}).
Intuitively, Ricci flow smooths the metric, but can lead to singularities that can be removed.  This procedure is known as surgery.   Ricci flow(and surgery) were used in an astonishing manner in the landmark work of Perleman \cite{perelman2002entropy} for solving the Poincar\'e conjecture, as well as in the proof of the differentiable sphere theorem by Simon Brendle and Richard Schoen \cite{spheretheo}.

One might image such powerful method can be applied to discrete geometry, where objects are irregular complex networks.
In 2019,  \cite{NLLG} claim good community detection on networks using Ricci flow defined on weighted graphs.  
The algorithm in \cite{NLLG} makes use of discrete Ricci flow based on Ollivier Ricci curvature which was introduced in \cite{Ollivier1, Ollivier} and  an analogous surgery procedure to partition networks which are modeled as weighted graphs.
The discrete Ricci flow deforms edge weights as time progresses:
edges of large positive Ricci curvature (i.e., sparsely traveled edges) will shrink and edges of very negative Ricci curvature (i.e., heavily traveled edges) will be stretched. By iterating the Ricci flow process, the heavily traveled edges are identified and thus communities can be partitioned. 
This approach has successfully detected communities for various networks including Zachary's Karate Club graph, Network of American football games, Facebook Ego Network, etc. The paper is beautiful in applications but lack of solid mathematical results/theorems. There are several fundamental questions needed to be addressed:
\begin{enumerate}
    \item What are intrinsic metric/curvature in graphs?
    \item Is the solution of Ricci-flow equation always exists? At what domain?
    \item If the limit object of the Ricci-flow exist?  Do they have constant curvature?
\end{enumerate}
The goal of this paper is to give a mathematical framework so that these essential questions can be answered rigorously.

We start with a brief explanation of the idea of discrete Ricci flow in \cite{NLLG} which contains two parts: the geometric meaning of Ollivier Ricci curvature and the role of Ricci flow.

In the setting of graphs, the Ollivier Ricci curvature is based on optimal transport of  probability measures associated to a lazy random walk. To generalize the idea behind Ricci curvature on manifolds to discrete space,  the spheres are replaced by probability measures $\mu_x, \mu_y$ defined on one-step neighborhood of vertices $x, y$. Measures will be transported by a distance equal to $(1- \kappa_{xy})d(x, y)$, where $\kappa_{xy}$ represents the Ollivier curvature along the geodesic segment $xy$. A natural choice for the distance between measures $\mu_x, \mu_y$ is the Wasserstein transportation metric $W_1$. Therefore, the Ollivier’s curvature is defined as: $\kappa_{xy} = 1- \frac{W_1(\mu_x, \mu_y)}{d(x, y)}$.
By this notion, positive Ollivier Ricci curvature implies that the neighbors of the two centers are close or overlapping, negative Ollivier Ricci curvature implies that the neighbors of two centers are further apart, and zero Ollivier Ricci curvature or near-zero curvature implies that the neighbors are locally embeddable in a flat surface. The Ollivier Ricci curvature can be generalized to weighted graph $(V, E, w)$ where $w$ indicates the edge weights. There has been various generalized versions, while the probability measure may vary the essence using optimal transport theory remains unchanged.

The discrete Ricci flow algorithm in \cite{NLLG} on weighted graphs is an evolving process. In each iteration, the Ricci flow process generate a time dependent family of weighted graph $(V, E, w(t))$ such that the weight $w_{ij}(t)$ on edge $ij$ changes proportional to the Ollivier Ricci curvature $\kappa_{ij}(t)$ at edge $ij$ at time $t$.
Ollivier\cite{Ollivier} suggested to use the following formula for Ricci flow with continuous time parameter $t$:
\begin{align}\label{equ:OllivierSuggest} 
\frac{d}{dt}w_{e}(t)= -\kappa_{e}(t) w_{e}(t),
\end{align}
where  $e\in E$,  $\kappa_e$ represents the Ollivier Ricci curvature on $e$ and $w_e$ indicates the length of edge $e$.
Then \cite{NLLG} uses the following formula for Ricci flow with discrete time $t$:
\begin{align}\label{equ:simpleRiccilfow}
    w_{ij}(t+1) = (1-\epsilon \kappa_{ij}(t)) d^t(i, j),
\end{align}
where  $d^{t}(i, j)$ is the associated distance at time $t$, i.e the shortest path length  between $i, j$ and $\kappa_{ij}(t)$ is the Ollivier Ricci curvature on edge $(i, j)$ at time $t$.  Observe such an iteration process, the Ricci flow enlarges the weights on negatively curvatured edges and shrink the weights on positively curvatured edges over time. By iterating the Ricci flow process, edges with high weights are detected so that can be removed by a surgery processes. As a result, the network is naturally partitioned into different communities with relatively large Ricci curvature.

Motivated by work in \cite{NLLG}, we propose a theoretic framework for the Ricci flow equations defined on weighted graphs. 
Since the Ricci  flow (\ref{equ:OllivierSuggest} ) does not preserve the sum of edge length of $G$, which would possibly lead to that the graph becomes infinitesimal in the limit if the initial metric satisfies a certain conditions, see an example in Section \ref{sec:examples}.  To avoid this,   we define the normalized Ricci flow:
\begin{align}\label{equ:ourRiccilfow}
    \frac{d}{dt} w_e(t)= -\kappa_e(t) w_e(t) +  w_e(t)  \dss_{h \in E(G)} \kappa_h(t)  w_h(t).
\end{align}
Here we adopt the Ollivier-Lin-Lu-Yau's Ricci curvature\cite{LLY}, which is the limit version of Ollivier Ricci curvature. Under this normalized flow, which is equivalent to the unnormalized Ricci flow (\ref{equ:OllivierSuggest}) by  scaling the metric in space by a function of $t$, the sum of edge length of the solution metric is $1$ in time. To see this, let $\vec{w}(t)=\{w_{e_0}(t), \ldots, w_{e_m}(t)\}$ be a solution of the unnormalized equation, let $\phi(t)$ be a function of time $t$ and $\phi(t)>0$. Set $\vec{\tilde{w}}(t)=\phi(t)\vec{w(t)}$ and $\sum_{e}\tilde{w_e}(t)=1$, then $\phi(t)=\frac{1}{\sum_{h}w_h(t)}$. Note the edge curvature $\kappa$ does not change under a scaling of the metric. Thus $\tilde{\kappa_e}=\kappa_e$ for all edges $e\in E$. Let $\tilde{t}=t$, then
\begin{align*}
      \frac{d}{d\tilde{t}}\tilde{w_e}(t)& =\frac{d\phi(t)w_e(t)}{dt} \frac{dt}{d\tilde{t}}=\Big(\frac{d\phi(t)}{dt}w_e(t) + \frac{dw_e(t)}{dt}\phi(t)\Big) \times 1\\
      &=-\frac{1}{(\sum_{h}w_h(t))^2} \sum_{h} \frac{dw_h}{dt}  w_e(t) -\kappa_e(t) w_e(t)\phi(t)\\
      &= \tilde{w_e}(t)  \dss_{h \in E(G)}\tilde{\kappa_h}(t)  \tilde{w_h}(t) -\tilde{\kappa_e}(t) \tilde{w_e}(t),
\end{align*}
where the last equation is obtained by replacing $w_e(t)$ by $\frac{1}{\phi(t)}\tilde{w_e}(t)$ for all edges $e$.

On the other side, let $\tilde{w}(t), w(t)$ be  solutions of the normalized equation and unnormalized equation respectively, we show that for each edge $e$, $w_e=\tilde{w_e}\sum_{h}w_h(t)$. It suffices to show that $\tilde{w_e}\sum_{h}w_h(t)$ satisfies equation (\ref{equ:OllivierSuggest}).
\begin{align*}
      \frac{d}{dt}\tilde{w_e}\sum_{h}w_h(t)& =\sum_{h}w_h(t)\frac{d}{dt}\tilde{w_e}(t) + \tilde{w_e}(t) \frac{d\sum_{h}w_h(t)}{dt} \\
      &=\sum_{h}w_h(t) \Big(\tilde{w_e}(t)  \dss_{h \in E(G)}\tilde{\kappa_h}(t)  \tilde{w_h}(t) -\tilde{\kappa_e}(t) \tilde{w_e}(t)\Big) +\tilde{w_e}(t)(-\dss_{h \in E(G)} \kappa_h(t)  w_h(t))\\
      &=-\kappa_e(t) w_e(t).
\end{align*}
Thus, there is a bijection between solutions of the unnormalized and normalized Ricci flow equations.

In Riemannian manifolds, with the establishment of  Ricci flow equations, one of important work is to verify whether this equation always has a  unique smooth solution, at least for a short time,  on any compact manifold of any dimension for any initial value of the metric.
In this paper, we study the problem of the existence and uniqueness of solutions and convergence results to the normalized Ricci flow (\ref{equ:ourRiccilfow}) on connected weighted graphs. The difficulty of the problem lies in that there is no explicit expression for $\kappa_e(t)$, although $\kappa_e$ can be written in terms of the infimum of distance-based 1-Lipschitz function,  there is no common optimal 1-Lipschitz function for all edges $e$, thus it is not easy to estimate the derivative of the right-hand side of Ricci flow.
The main Theorem \ref{thm:sol_continuous} of this paper proves the long-time existence and uniqueness of solutions for the initial value problem involved Ricci flow equations (\ref{equ:ourRiccilfow}) provided that each edge is always the shortest path connecting its endpoints over time.  This theorem also implies a same result for the unnormalized Ricci flow (\ref{equ:OllivierSuggest}). 
We also prove that several convergence results of Ricci flow on path and star graphs. Our results display different possible solutions of Ricci flow for the path of length 2, a graph minor of any finite path resulted by the Ricci flow accompanied with edge operations (see Theorem \ref{thm:pathconverge}) and prove that Ricci flow on star graph can deforms any initial metric to a constant-curvatured metric (see Theorem \ref{thm:starconverge}).

The paper is organized as follows. In section \ref{sec:notaions}, we introduce the notion of Ollivier-Lin-Lu-Yau Ricci curvature defined on weighted graphs and related lemmas; in section \ref{sec:Ricciflow}, we introduce the Ricci flow equation and prove our main theorem;  in section \ref{sec:SolutionstothecontinuousRicciflow}, we display different types of solution for continuous Ricci flows on path graph; in section \ref{sec:ConvergenceofRicciflow}, we prove convergence results of normalized Ricci flow on path and star.

\section{Notations and lemmas}\label{sec:notaions}
Let $G=(V, E, w)$ be a weighted graph on vertex set $V$ associated by the edge weight function $w:E\to [0, \infty)$. 
For any two vertices $x, y$, we write $xy$ or $x\sim y$ to represent an edge $e=(x, y)$, $w_{xy}$ is always positive if $x\sim y$. For any vertex $x\in V$, denote the neighbors of $x$ as $N(x)$ and the degree of $x$ as $d_x$. The length of a path is the sum of  edge lengths on the path, for any two vertices $x, y$,  the distance $d(x, y)$ is the length of a minimal weighted path among all paths that connect $x$ and $y$.
We call $G$ a combinatorial graph if $w_{xy}=1$ for $x\sim y$, $w_{xy}=0$ for $x\not\sim y$.
Next we recall the definition of Ricci curvature defined on weighted graphs.

\begin{definition}\label{probabilitydistribution}
 A probability distribution over the vertex set $V(G)$ is a mapping $\mu: V\to [0,1]$ satisfying $\sum_{x \in V} \mu (x)=1$. Suppose that two probability distributions $\mu_1$ and $\mu_2$ have finite support. A \textit{coupling} between $\mu_1$ and $\mu_2$ is a mapping $A: V\times V\to [0, 1]$ with finite support such that
$$\sum\limits_{y \in V} A(x, y)=\mu_1(x) \ \text{and} \sum\limits_{x \in V} A(x, y)=\mu_2(y). $$
\end{definition}

\begin{definition}
The \textit{transportation distance} between two probability distributions $\mu_1$ and $\mu_2$ is defined as follows:
$$W(\mu_1, \mu_2)=\inf_{A} \sum_{x, y\in V} A(x, y)d(x, y),$$
where the infimum is taken over all coupling $A$ between $\mu_1$ and $\mu_2$.
\end{definition}

By the theory of linear programming, the transportation distance is also equal to the optimal solution of its dual problem. Thus, we also have \[W(\mu_1, \mu_2)=\sup_{f} \sum_{x\in V}f(x)[\mu_1(x)-\mu_2(x)]\] where $f$ is $1$-Lipschitz function satisfying \[|f(x)-f(y)| \leq d(x, y) \ \textrm{for $\forall x, y\in V(G)$ }.\]

\begin{definition} \cite{Ollivier}\cite{LLY}\cite{blhy}\label{def:alpharicci}
Let $G=(V, E, w)$ be a weighted graph where the distance $d$ is determined by the weight function $w$.  For any $x, y\in V$ and $\alpha \in [0, 1]$,  the \textit{$\alpha$-Ricci curvature} $\kappa_{\alpha}$ is defined to be
$$\kappa_{\alpha}(x, y)=1-\frac{W(\mu_x^{\alpha}, \mu_y^{\alpha})}{d(x, y)},  $$
where the probability distribution $\mu_x^{\alpha}$ is defined as:
\[\mu_x^{\alpha}(y)=
\begin{cases}
 \alpha,  & \text{if $y=x$}, \\
 \displaystyle (1-\alpha)\frac{\gamma(w_{xy})}{\sum_{z\sim x} \gamma(w_{xz})},  & \text{if $y\sim x$},\\
 0,& \text{otherwise},
\end{cases}\]
where $\gamma: \mathbb{R}_{+} \to \mathbb{R}_{+}$ represents an arbitrary one-to-one function which guarantees that the curvature is invariant under metric scaling. 
The Lin-Lu-Yau's Ollivier \textit{Ricci curvature} $\kappa(x, y)$ is defined as
\[\kappa(x, y)=\lim\limits_{\alpha\to 1} \frac{\kappa_{\alpha}(x, y)}{1-\alpha}. \]
\end{definition}
It is clear that curvature $\kappa$ is a continuous function on values of weight  $w$. 
On combinatorial graphs,  the probability distribution $\mu_x^{\alpha}$ is uniform on $x$'s neighbors, the above limit expression for Lin-Lu-Yau’s Ollivier curvature is studied in \cite{LLY, BCLMP} and it turned out that function $\kappa_{\alpha}:\alpha \to \mathbb{R}$ is a piece-wise linear function with at most three pieces. Therefore one can calculate $\kappa$ easily by choosing a large enough value of $\alpha$.  On weighted graphs, the probability distribution $\mu_x^{\alpha}$ is determined by weight $w$ and function $\gamma$,  the distance $d$ involved is reflected directly by $w$. Some authors used a combinatorial distance $d$ which measures the number of edges in the shorted path instead of the weighted version. For instance, in  \cite{MW},
the authors \cite{MW} also simplify the limit expression of $\kappa(x, y)$ to two different limit-free expressions via graph Laplacian and via transport cost. Although the details are different, the curvature definitions are essentially the same.
The limit-free expression of $\kappa(x, y)$ in \cite{MW} is still true under our definition.

To state this limit-free curvature expression, we need to rephrase the notion of Laplacian in order to adapt to Definition \ref{def:alpharicci}.
Let $G=(V, w, \mu)$ be a weighted graph, let $f$ represent a function in $\{f: V \to \mathbb{R}\}$. The {\it gradient} of $f$ is defined by
$$\nabla_{xy}f=\frac{f(x)-f(y)}{d(x,y)}\ \ \mbox{for $x\neq y$. }$$
According to Definition \ref{def:alpharicci}, the graph {\it Laplacian} $\Delta$ is defined via:
\begin{align}\label{equ:laplacianformula}
    \Delta f(x)= \frac{1}{\sum_{z\sim x} \gamma(w_{xz})}\sum_{y\sim x} \gamma(w_{xy})(f(y)-f(x)),
\end{align}
where $f\in\{f: V \to \mathbb{R}\}$.

The  limit-free formulation of the Lin-Lu-Yau Ricci curvature using graph Laplacian and gradient is as follows. 

\begin{theorem}\label{thm:curvature_laplacian}\cite{MW} (Curvature via the Laplacian) Let $G= (V,w,m)$ be a weighted graph and let $x \neq y \in V(G)$. Then

\begin{equation*}
    \kappa(x,y) = \inf_{\substack{f \in Lip(1)\\ \nabla_{yx}f = 1}} \nabla_{xy} \Delta f,
\end{equation*}
\end{theorem}

where $\nabla_{xy}f$ is the gradient of $f$,  $d$ is the combinatorial graph distance.

\begin{remark}
Although in Theorem \ref{thm:curvature_laplacian}, $\kappa(x,y)$ is defined with the assumption that $d$ is the usual combinatorial graph distance, however, the proof of Theorem \ref{thm:curvature_laplacian} works verbatim when $d$ is the weighted distance. Please refer to the detailed proofs in \cite{MW}. 
\end{remark}

 Using the weighted distance in Definition \ref{def:alpharicci}, the limit expression is  simplified to another limit-free version via a so called $\ast$-coupling functions\cite{blhy}.  
Let $\mu_x:=\mu_x^0$ be the probability distribution of random walk at $x$ with idleness equal to zero. For any two vertices $u$ and $v$, a {\em $\ast$-coupling} between $\mu_u$ and $\mu_v$ is a mapping $B: V\times V\to \mathbb{R}$ with finite support such that 
\begin{enumerate}
    \item $0<B(u,v)$, but all other values $B(x,y)\leq 0$.
    \item $\sum\limits_{x,y\in V} B(x,y)=0$.  
    \item $\sum\limits_{y \in V} B(x, y)=-\mu_u(x)$ for all $x$ except $u$. 
    \item $\sum\limits_{x \in V} B(x, y)=-\mu_v(y)$ 
for all $y$ except $v$.  
\end{enumerate}
Because of items (2),(3), and (4), we get
$$B(u,v)=\sum_{(x,y)\in V\times V \setminus \{(u,v)\}}-B(x,y) \leq \sum_{x} \mu_u(x) + \sum_{y} \mu_v(y)\leq 2.$$

 \begin{theorem}(Curvature via Coupling function)\cite{blhy}\label{thm:curvatureviacoupling}
 Let $G= (V,w,m)$ be a weighted graph and let $u, v\in V(G)$ and $u \neq v $. Then
$$\kappa(u,v)=\frac{1}{d(u,v)}\sup\limits_{B} \sum\limits_{x, y\in V} B(x, y)d(x, y),$$
where the superemum is taken over all weak $\ast$-coupling $B$ between $\mu_u$ and $\mu_v$.
\end{theorem}

Since $0<B(u,v)$ and $B(x,y)\leq 0$.
then $$\kappa(u,v)\leq \frac{1}{d(u,v)}B(u,v)d(u, v)\leq 2.$$ 

A lower bound of $\kappa(u,v)$ is obtained by using a result of Lemma 3.2 in \cite{blhy}.  
We re-organized this result as follows: 
\begin{lemma}\label{coro:curvaturebound}
Let $G=(V, E)$ be a weighted graph associated by a edge weight  function $w$ where the maximum value of $w$ is denoted by $D(G)$.  
Let $u \neq v \in V$. Then
$$\kappa(u,v)\geq -\frac{2D(G)}{d(u, v)}.$$
\end{lemma}

\begin{proof}
For any vertex $u\in V$, let $D_u=\sum_{x\in N(u)} \gamma(w_{ux})$. 
Fix an edge $uv\in E(G)$, 
we define a function $B: V\times V\to \mathbb{R}$ for calculating $\kappa(u, v)$. 
For any $x\in N(u)\setminus \{v\}$, let 
$B(x,y)=-\frac{\gamma(w_{ux})}{D_u}$ if $y=v$ and 
 $0$ otherwise.
For any $y\in N(v)\setminus \{u\}$, let
$B(x,y)=-\frac{\gamma(w_{vy})}{D_v}$ if $x=u$ and 
$0$ otherwise.
Let $B(v,v)=-\frac{\gamma(w_{uv})}{D_u}$,
$B(u,u)=-\frac{\gamma(w_{uv})}{D_v}$,
and $B(u,v)=2$.
The rest of entries are set to $0$.
It is straightforward to verify the following results: 
\begin{center}
 $\sum\limits_{x,y\in V} B(x,y)=0$;  $\sum\limits_{y \in V} B(x, y)=-\mu_u(x)$ for all $x$ except $u$;  $\sum\limits_{x \in V} B(x, y)=-\mu_v(y)$ for all $y$ except $v$.
\end{center} 
Thus $B$ is $\ast$-coupling between $\mu_u$ and $\mu_v$. 
By Theorem \ref{thm:curvatureviacoupling}, we have 
\begin{align}
\begin{split}
\kappa(u,v) &\geq \frac{1}{d(u,v)} \sum_{x,y\in V}B(x,y)d(x,y)\\
&= 2
- \sum_{x\in N(u)\setminus \{v\}} \frac{\gamma(w_{ux})}{D_u} 
\frac{d(x,v)}{d(u,v)}
-\sum_{y\in N(v)\setminus \{u\}} \frac{\gamma(w_{vy})}{D_v} 
\frac{d(u, y)}{d(u,v)}\\ 
&\geq 2
- \sum_{x\in N(u)\setminus \{v\}} \frac{\gamma(w_{ux})}{D_u} 
\frac{d(x,u)+d(u,v)}{d(u,v)}
-\sum_{y\in N(v)\setminus \{u\}} \frac{\gamma(w_{vy})}{D_v} 
\frac{d(y,v)+d(u,v)}{d(u,v)}\\ 
&= 2
- \sum_{x\in N(u)\setminus \{v\}} \frac{\gamma(w_{ux})}{D_u} 
-\sum_{y\in N(v)\setminus \{u\}} \frac{\gamma(w_{vy})}{D_v} 
-
\sum_{x\in N(u)\setminus \{v\}} \frac{\gamma(w_{ux})}{D_u} 
\frac{d(x,u)}{d(u,v)}\\
&\quad\quad
-\sum_{y\in N(v)\setminus \{u\}} \frac{\gamma(w_{vy})}{D_v} 
\frac{d(y,v)}{d(u,v)}\\ 
&= \frac{w_{uv}}{D_u} + \frac{\gamma(w_{uv})}{D_v} 
- \sum_{x\in N(u)\setminus \{v\}} \frac{\gamma(w_{ux})}{D_u} 
\frac{d(x,u)}{d(u,v)}
-\sum_{y\in N(v)\setminus \{u\}} \frac{\gamma(w_{vy})}{D_v} 
\frac{d(y,v)}{d(u,v)}\\
&= \frac{2\gamma(w_{uv})}{D_u} + \frac{2\gamma(w_{uv})}{D_v} 
- \sum_{x\in N(u)} \frac{\gamma(w_{ux})}{D_u} 
\frac{d(x,u)}{d(u,v)}
-\sum_{y\in N(v)} \frac{\gamma(w_{vy})}{D_v} 
\frac{d(y,v)}{d(u,v)}.
\end{split}
\end{align}

Let $D(G)$ denote the maximum edge length of $G$, i.e. $d(x, y)\leq D(G)$ for all $x\sim y$. Then 
\begin{align}\label{inequal:lowerofkuv}
\begin{split}
\kappa(u,v)
&\geq - \frac{D(G)}{d(u,v)} \Big( \sum_{x\sim u} \frac{\gamma(w_{ux})}{D_u}  + \sum_{y\sim v} \frac{\gamma(w_{vy})}{D_v} \Big)= - \frac{D(G) }{d(u,v)} \times 2.
\end{split}
\end{align}
\end{proof}

\section{Continuous Ricci flow process} \label{sec:Ricciflow}
In this section, we will describe a continuous Ricci flow process on weighted graphs and prove that this Ricci flow has a unique solution that exists for all time. 

Let $\kappa: E(G) \to \mathbb{R}^{|E(G)|}$ be the Ollivier-Lin-Lu-Yau curvature function defined on a weighted graph $G = (V,w,\mu)$ where $w$ is the weight function on the edge set of $G$ and $\mu = \{\mu_x: x\in V(G)\}$ be probability distribution function such that for each $x \in V(G)$,
\begin{equation}\label{eq:weight_dist}
\mu^{\alpha}_x(y) = \begin{cases}
    \alpha & \textrm{ if $x=y$},\\
 (1-\alpha) \frac{\gamma(w_{xy})}{\sum_{z\in N(x)} \gamma(w_{xz})} & \textrm{ if $y\in N(x)$},\\
0 & \textrm{ otherwise},
\end{cases}
\end{equation}
where $\gamma:\mathbb{R}_{+} \to \mathbb{R}_{+}$ is a Lipschitz function over $[\delta, 1]$ for all $\delta > 0$.

Let $X(t) = \langle w_{e_1}(t), w_{e_2}(t),\ldots, w_{e_m}(t)\rangle \in \mathbb{R}_+^m$  where $t\in [0,\infty)$ and $m = |E(G)|$ denotes the number of edges of $G$. 
Since  the initial weight on each edge is not zero, we  take    $X_0 \in \mathbb{R}_{++}^m$, an arbitrary vector $\langle w_{e_1}(0), w_{e_2}(0), \ldots, w_{e_m}(0)\rangle$ with $\sum_{i=1}^{m} w_{e_i}(0) = 1$,  as the initial metric $\vec{w}(0)$ for graph $G$. Then, we define a system of ordinary differential equations as follows:
\begin{equation}\label{eq:stochastic_continuous}
\begin{cases}
    \frac{dw_e}{dt} = -\kappa_e w_e + w_e \sum_{h \in E(G)} \kappa_h w_h\\
     X(0) = X_0.
\end{cases}
\end{equation}

Now we introduce the continuous Ricci flow process as follows: 

\begin{algorithm}[H]\label{alg:continuousricciflow}
\KwIn{An undirected graph $G$, merge threshold $mt>0$, and termination threshold $\delta > 0$.}
\KwOut{A collection of vertex-disjoint minors of $G$, which are the `clusters' if $G$ is viewed as a network.}

Set hierarchy level to be 1.

Let $\vec{w}$ be the solution to the system of ODE described in \eqref{eq:stochastic_continuous} with the exit condition:\linebreak
 (I)  $w_{uv}(t) > d(u, v)(t)$ for some $uv\in E(G)$ and some $t \in [0,\infty)$; \linebreak
 (II) $w_{uv}(t) < mt$ for some $uv\in E(G)$ and some $t \in [0,\infty)$;\linebreak
 If Condition (I) is met, delete the edge $uv$ and restart step $2$;\linebreak
If Condition (II) is met, contract the edge $uv$ and restart step $2$.\linebreak
 If $\vec{w}$ is a non-chaotic solution, go to Step $3$. \linebreak
 If $\vec{w}$ is chaotic, slightly perturb the initial conditions.
 
 Let $G'$ be the resulting graph from Step 2. \linebreak
(I) Label each edge of $G'$ with the current hierarchy level.\linebreak
(II) Increase the hierarchy level by 1. \linebreak
 Perform Step $2$ on $G'$.
\caption{Continuous Ricci flow process}
\end{algorithm}

We make three observations: 

First, there is no edge getting a zero weight at any time $t$ during the whole Ricci flow process. 
In the initial weighted graph $G$, $w_e(0)\neq 0$ for all edges $e\in E(G)$. Fix an edge $e$, the right-hand side of (\ref{eq:stochastic_continuous}) is bounded below by $w_e(-2-2|E(G)|)$ according to Lemma \ref{coro:curvaturebound}, then $w_e(t)>w_e(0)e^{(-2-2|E(G)|)t}$ which is always positive at finite time.

Second, each edge has a weight assigned to it over time, theoretically, we don't know if there is an edge meeting the exit condition (I), that is,  $w_{uv}(t) > d(u, v)(t)$ for some $uv\in E(G)$.  To fix this possible barrier, we choose to delete such edges, notice that the resulting graph is still connected. In addition, the only reason for the reduction in the number of vertices is the exit condition (II). Thus, $G$ will not degenerate to a point. 

Last, since $\sum_{i=1}^{m} w_{e_i}(0) = 1$,  under the assumption that no edges meeting two exit conditions, we claim that
the property $\sum_{i=1}^{m} w_{e_i} = 1$ is always maintained. 
To see this, 
let $T(t) = \sum_{h\in E(G^t)} w_h(t)$, where $G^t$ is the resulting graph at time $t$. Sum up both sides of equation (\ref{eq:stochastic_continuous}) over all edges of $G^t$, we have
$$
\frac{d T(t)}{dt} =(T(t)-1)\sum_{h \in E(G^t)} \kappa_h(t) w_h(t).
$$
By Theorem \ref{thm:curvature_laplacian}, the right hand side is a bounded value for all $t$, 
it follows then that $T(t)-1$ has the following form:
$$ T(t)-1 =ce^{\ds\int_{0}^t (\sum_{h\in E(G^s)} \kappa_h(s) w_h(s)) ds},$$ where $c$ is a constant depending on $T(0)$.
Since $T(0) = 1$, then $c=0$ implies $T(t)=1$ for all $t\geq 0$, done. 
Therefore, in algorithm \ref{alg:continuousricciflow}, for all time $t$ the sum of weight is at most $1$, and for each edge $e$, $0<w_e(t)\leq 1$. In order to remain the sum of weight constant $1$, an alternative approach is re-normalize the edge weight after each Ricci flow iteration so that the sum of weight always remains $1$, but sum of weight being $1$ or not does not affect the validity of the following theorem.

\subsection{Existence and uniqueness of the solution}

\begin{theorem}\label{thm:sol_continuous}
For any initial weighted graph $G$,  by fixing the exit condition (I),  
there exists a unique solution $X(t)$, for all time $t\in [0, \infty)$,  to the system of ordinary  differential equations in \eqref{eq:stochastic_continuous}. 
\end{theorem}

Before we prove Theorem \ref{thm:sol_continuous}, we first need some lemmas. By the exit condition (I) stated in above algorithm, once $w_{uv}(t) > d(u, v)(t)$ for some $uv\in E(G)$ we will delete the  edge $uv$, thus $w_{xy}$ always represent the length of edge $uv$. For convenience,  we use $w_{xy}$ instead of $d(x, y)$ to represent the distance between any pair of vertices $x$ and $y$.

\begin{lemma}\label{lem:lip-extension}
Let \( G = (V, E) \) be a connected graph with two metrics:
\( \vec{w} >0\) and \( \vec{w}' >0\) such that \( \| \vec{w}' - \vec{w} \|_\infty < \epsilon \).
Let \( f \colon V \to \mathbb{R} \) be 1-Lipschitz under \( \vec{w} \). 
Then there exists a 1-Lipschitz function under $\vec{w'}$ such that $|f'(z)-f(z)|\leq \epsilon$ for every $z\in V$. 
\end{lemma}
\begin{proof}

Define \( f' \colon V \to \mathbb{R} \) as:
\[
f'(y) = \inf_{u \in V} \left\{ f(u) + w'_{uy} \right\} \quad \forall y\in V,
\]

For any vertex \( u \), we have $
w'_{uy} \leq w'_{ux} + w'_{xy}.$
Then
\[
f(u) + w'_{uy} \leq f(u) + w'_{ux}+ w'_{xy}.
\]
The infimum preserves the inequality, that is
\begin{align*}
f'(y) &= \inf_{u \in V} \left\{ f(u) + w'_{uy} \right\} \\
&\leq \inf_{u \in V} \left\{ f(u) + w'_{ux} + w'_{xy} \right\}  \\
&= \left( \inf_{u \in V} \left\{ f(u) + w'_{ux} \right\} \right) + w'_{xy} \\
&= f'(x) + w'_{xy}.
\end{align*}
Thus, $f'$ is a 1-Lipschitz function under the metric $\vec{w'}$.

Since \( \| \vec{w}' - \vec{w} \|_\infty < \epsilon \), for any pair $xy$ 
\[
w_{xy} - \epsilon\leq w'_{xy} \leq w_{xy} + \epsilon.
\]
Take \( u=y \) in the infimum:
$$
f'(y) = \inf_{u \in V} \{ f(u) + w'_{uy} \} 
\leq f(y) + w'_{yy} \leq  f(y).
$$
Let $u$ be the minimizer of $f'(y)$, that is $f'(y) = f(u) + w'_{uy}$. Since \( f \) is 1-Lipschitz under \( \vec{w} \), then $
f(u) \geq f(y) - w_{uy}.
$
Thus
\[
f(u) + w'_{uy} \geq f(y) - w_{uy} + w'_{uy}\geq f(y) - \epsilon,
\]
which gives

\[
f'(y) = f(u) + w'_{uy}  \geq f(y) -  \epsilon.
\]
\end{proof}

The following lemmas guarantee the existence of a 1-Lipschitz function that is tight for a specified pair of vertices.

\begin{lemma}\label{lem:near-lip}
Let $G=(V, E, w)$ be a weighted graph  and $x,y$ be two fixed vertices in $G$. For any $1$-Lipschitz function $f$ defined on $G$ and such that $0 <f(y)-f(x)< w_{xy}$, there exists a $1$-Lipschitz function $f'$, such that $f'(y)-f'(x) = w_{xy}$ and $f'(z) -f(z) \leq  w_{xy}-(f(y)-f(x))$ for all $z \in V$. 
\end{lemma}

\begin{proof}
Define function $g$ on $G$ such that
\begin{equation*}
    \begin{cases}
  g(y) = f(x) + w_{xy}, \\
 g(z) = f(z) & z\in V\setminus\{y\}. 
\end{cases}
\end{equation*}
We have 
$$g(y)-g(x) = w_{xy},$$
$$ g(y) - f(y) = w_{xy}-(f(y)-f(x)), $$
$$g(z)-f(z)=0\leq w_{xy}-(f(y)-f(x)) \ \forall z\neq y.$$
Thus,  if $g$ is $1$-Lipschitz on $G$, let $f'=g$,  we are done. 

If $g$ is not $1$-Lipschitz,  according to its definition, it then fails at vertex $y$ and some vertex in $V\setminus \{x, y\}$, denote such vertices as $v_{1}, v_{2},\ldots, v_{t}$ in the order that $g(v_{1}) \leq g(v_{2}) \leq \cdots \leq g(v_{t})< g(y)$.
Denote set $N=\{y, v_{1}, v_{2},\ldots, v_{t} \}$. We have  $|g(v_j)-g(y)| >w_{yv_j}$ for each $j\leq t$. That is,  either $g(v_j)- g(y)>w_{yv_j}$ or $g(v_j)- g(y)< -w_{yv_j}$. Note $g(v_j)- g(y)=g(v_j)-f(x)-w_{xy} \leq w_{xv_j}-w_{xy}\leq w_{yv_j}$,  thus, it has to be the latter case, 
 that is,  $g(v_j)- g(y)< -w_{yv_j}$ for all $j\leq t$.

Note at any pair of vertices out of $N$, $g$ is 1-Lipschitz. Further,  at vertex $x$ and vertex $v_j$, we have $g(v_j)-g(x)=g(v_j)-g(y)+w_{xy}<-w_{yv_j}+w_{xy}\leq w_{xv_j}$, thus,  $g(v_j)-g(x)$ is strictly less than $w_{xv_j}$.

Now we  create a new function $g'$ from $g$ so that $g'$ is  1-Lipschitz on $G$. Let
\begin{equation*}
    \begin{cases}
  g'(v_{j})=g(v_{j})+a_j & j\leq t, \\
  g'(u)=g(u) & \text{otherwise}, 
\end{cases}
\end{equation*}
where values $a_j$ will be chosen from internal 
$$L_j=\Big[g(y)-g(v_{j})-w_{yv_j}, \min\{w_{xy}-(f(y)-f(x)), \  \min_{z\not\in N}\{w_{zv_j} +g(z)- g(v_j)\}\}\Big].$$ One can check that internal $L_j$ is non empty and $a_j$ is positive.  
 Further, let $a_j$ satisfy 
 $0<a_j-a_i<g(v_i)-g(v_j)+w_{v_{i}v_{j}}$ for all $1\leq j<i\leq t$. Note we are able to achieve this purpose by choosing value of $a_j$ as large as possible in the reverse order (i.e., from $a_t$ to $a_1$). Next, we will confirm that $g'$ is  1-Lipschitz on $G$. 
\begin{itemize}
    \item For $y$ and each $v_j$, $g'(v_{j})-g'(y)=g(v_{j})+a_j-g(x)-w_{xy}<g(v_{j})+w_{xy}-f(y)+f(x)-g(x)-w_{xy}<g(v_{j})-f(y)<w_{yv_j}$, thus, $g'$ is 1-Lipschitz at $y$ and $v_j$.
    \item For $v_{i}$ and $v_{j}$, $j<i$, 
    $g'(v_{j})-g'(v_i)=g(v_{j})-g(v_i)+a_{j}-a_i<g(v_{j})-g(v_i)+ g(v_i)-g(v_j)+w_{v_iv_{j}}<w_{v_{i}v_{j}}$
and $g'(v_{j})-g'(v_i)=g(v_{j})-g(v_i)+a_{j}-a_i>g(v_{j})-g(v_i)>-w_{v_jv_i}$.  Thus, $g'$ is 1-Lipschitz at $v_i$ and $v_j$.
\item  For  $z\not\in N$ and each $v_j$'s.   $g'(v_j)-g'(z)=g(v_j)+a_j-g(z)\leq g(v_j)-g(z)+ w_{zv_j} +g(z)-g(v_j)\leq w_{zv_j}$ and $g'(v_j)-g'(z)=g(v_j)+a_j-g(z)>g(v_j)-g(z)>-w_{zvj}$.
Thus, function $g'$ is  1-Lipschitz at  $z$ and $v_j$'s. 

\item For $u, v\notin N$, $|g'(u)-g'(v)|=|g(u)-g(v)|\leq w_{uv}$. Thus, function $g'$ is 1-Lipschitz out of $N$.

\end{itemize}

To sum up, there exist positive values $a_i$ so that $g'$  is 1-Lip between all pairs of vertices of $G$. 
One can also check that $g'$ satisfies inequalities stated in the lemma: 
 $$g'(y)-g'(x) = f(y)+w_{xy}-g(x)=w_{xy}, $$ 
    $$g'(v_i) - f(v_i) =a_i \leq  w_{xy}-(f(y)-f(x)), \ j\leq t, $$
    $$g'(z) -f(z)= 0<  w_{xy}-(f(y)-f(x)) \ z\notin N,
   $$ 
  Let $f'=g'$, the proof is complete.

\end{proof}

A very similar proof gives the following result:
\begin{lemma}\label{lem:near-lip2}
Let $G=(V, E, w)$ be a weighted graph  and $x,y$ be two fixed vertices in $G$. Let $0< a <w_{xy}$. For any $1$-Lipschitz function $f'$ defined on $G$  such that $f'(y)-f'(x) =w_{xy}$, there exists a $1$-Lipschitz function $f$, such that $f(y)-f(x)=a$ and $|f'(z) -f(z)|\leq w_{xy}-a$ for all $z \in V$.
\end{lemma}


\begin{proof}
Define function $g$ on $G$ such that
\begin{equation*}
    \begin{cases}
  g(y) = f'(x) + a, \\
 g(z) = f'(z) & z\in V\setminus\{y\}. 
\end{cases}
\end{equation*}
We have 
$$g(y)-g(x) =a,$$
$$ g(y) - f'(y) = w_{xy}-a, $$
$$g(z)-f'(z)=0\leq w_{xy}-a \ \forall z\neq y.$$
Thus,  if $g$ is $1$-Lipschitz, let $f=g$,  we are done.

If $g$ is not $1$-Lipschitz, then there exists $z\neq x, y$ so that  $|g(z)-g(y)| >w_{yz}$ and one can check that it is the case $g(z)- g(y)>w_{yz}$.
Denote such vertices as $v_{1}, v_{2},\ldots, v_{t}$ in the order that $g(v_{1}) \geq g(v_{2}) \geq \cdots \geq g(v_{t})> g(y)$. Denote set $N=\{y, v_{1}, v_{2},\ldots, v_{t} \}$.
Observe that $g$ is 1-Lip for pair of vertices not in $N$ and we have $|g(v_i)-g(x)|$ not equal to $w_{xv_i}$, as $g(v_i)-g(x)>w_{yv_i}+g(y)-g(x) = w_{yv_i}+a >-w_{xv_i}$ and $g(v_i)-g(x)<-w_{yv_i}+a<-w_{yv_i}+w_{xy}<w_{xv_i}$. Thus adding an appropriate negative value to $g(v_i)$ will not affect the pair $x$ and $v_i$.


Now we  create a new function $g'$ from $g$ so that $g'$ is  1-Lipschitz on $G$. Let 
\begin{equation*}
    \begin{cases}
  g'(v_{j})=g(v_{j})-a_j & j\leq t, \\
  g'(u)=g(u) &  \text{otherwise}, 
\end{cases}
\end{equation*}
where values $a_j$ will be chosen from internal 
$$L_j=\Big[\max\{g(v_{j})-g(y)-w_{yv_j}, g(v_i)-g(x)-w_{xv_i}\}, w_{xy}-a \Big].$$ One can check that internal $L_j$ is non empty and $a_j$ is positive. 
 Further, let $a_j$ satisfy 
 $0<a_j-a_i<g(v_j)-g(v_i)+w_{v_{i}v_{j}}$ for all $1\leq j<i\leq t$, note we are able to achieve this purpose by choosing value of $a_j$ as large as possible in the reverse order (i.e., from $v_t$ to $v_1$). Next, we will confirm that $g'$ is 1-Lipschitz  on $G$. 
\begin{itemize}
    \item For $y$ and each $v_j$, $g'(v_{j})-g'(y)=g(v_{j})-a_j-g(y)<g(v_{j})-g(v_{j})+g(y)+w_{yv_j}-g(y)<w_{yv_j}$, $g'(v_{j})-g'(y)=g(v_{j})-a_j-g(y)>f'(v_j)-w_{xy}+a-f'(x)-a>w_{xv_j}-w_{xy}>-w_{yv_j}$, thus, $g'$ is 1-Lipschitz at $y$ and $v_j$.

   \item  For $v_{i}$ and $v_{j}$, $j<i$,
    $g'(v_{j})-g'(v_i)=g(v_{j})-g(v_i)-(a_{j}-a_i)<g(v_{j})-g(v_i)\leq w_{v_{i}v_{j}}$, 
and $g'(v_{j})-g'(v_i)=g(v_{j})-g(v_i)-(a_{j}-a_i)>-w_{v_jv_i}$.  Thus, $g'$ is 1-Lipschitz at $v_i$ and $v_j$.
\item 
For  $z\not\in N$ and each $v_j$'s,    $g'(v_j)-g'(z)=g(v_j)-a_j-g(z)< g(v_j)-g(z)\leq w_{zv_j}$ and $g'(v_j)-g'(z)=g(v_j)-a_j-g(z)>g(v_j)-g(z) -g(v_j)+g(x)+w_{xv_j} \geq -w_{xz}+w_{xv_j}>-w_{zvj}$.
Thus, function $g'$ is still 1-Lipschitz at  $z$ and $v_j$'s. 
\item For $u, v\notin N$, $|g'(u)-g'(v)|=|g(u)-g(v)|\leq w_{uv}$. Thus, function $g'$ is 1-Lipschitz out of $N$.

\end{itemize}

To sum up, there exist positive values $a_i$ so that function $g'(v)$ obtained from $g$ by reducing values $\{a_i\}$ is 1-Lipschitz between all pairs of vertices of $G$,  and $g'(y)-g'(x) =a$,  $|g'(z) -f(z)| <  w_{xy}-a$  for all $z \in V$ are satisfied. Let $g'=f$, the proof is complete.

\end{proof}

In order to show Theorem \ref{thm:sol_continuous}, we need some classical theorem on the existence and uniqueness of solutions to a system of ordinary differential equations.

\begin{lemma}\label{lem: existenceunique_ODE}\cite{bookODE}
Suppose that vector-valued  function $F(t, X)=\{f_1(t,X), \cdots, f_n(t, X)\}$ is continuous in some $n+1$ dimensional region:
$$R=\{(t,X): |t|\leq a, \left\lVert X-X_0\right\lVert \leq b\}, $$ and is Lipschitz continuous about $X=(x_1, x_2, \cdots,x_n)$.
Then the the following ODE's initial value problem 
$$\frac{d X}{dt}=F(t, X), \hspace{5mm} X(0)=X_0$$
 has a unique solution $X=X(t)$ at region $|t|\leq \alpha$, where 
 $$\alpha=\min\{a, \frac{b}{N}\}, \hspace{5mm} N=\max_{(t, X)\in R}\ \left\lVert F(t, X)\right\lVert. $$
\end{lemma}
A vector-valued function satisfies a continuous or a Lipschitz condition in a region  if and only if its component functions satisfy these conditions in the same region. 
The following theorem is used to estimate the maximum existence interval of solutions of the following initial value problem:
\begin{align}\label{equ:sigeloed}
    \frac{dy}{dx} = f(x, y), \hspace{5mm} y(x_0)=y_0. 
\end{align}

For narrative convenience, we  call a function $\phi(x)$ as right-top solution to (\ref{equ:sigeloed}) if 
$$\frac{d\phi}{dx} > f(x, y), \hspace{5mm} \phi(x_0)\geq y_0. $$
And we  call a function $\Phi(x)$ as right-bottom solution to (\ref{equ:sigeloed}) if 
$$\frac{d\Phi}{dx} < f(x, y), \hspace{5mm} \Phi(x_0)\leq y_0. $$

\begin{lemma}\cite{bookODE}
Suppose $f(x, y)$ is a continuous function in the region $R=\{(x, y), x_0\leq x<b, -\infty< y< \infty\}$, and $(x_0, y_0)\in R$. Denote $[x_0, \beta_1)$ as the maximum existence interval of solution to (\ref{equ:sigeloed}). If (\ref{equ:sigeloed}) has right-top solution $\phi(x)$ and right-bottom solution $\Phi(x)$ and they have the same interval of solutions $[x_0, \beta)$, then $\beta_1\geq \beta$. 
\end{lemma}

Now, we are ready to prove Theorem \ref{thm:sol_continuous}.

\begin{proof}[Proof of Theorem \ref{thm:sol_continuous}]

For a fixed $\delta > 0$, define
$$S = \{\langle w_1, w_2, \ldots, w_m \rangle: w_i > 0, \sum_{i\in [m]} w_i \leq  1\}$$
and
$$S_{\delta} = \{\langle w_1, w_2, \ldots, w_m \rangle: w_i \geq \delta, \sum_{i\in [m]} w_i \leq   1\}.$$
We first show that \eqref{eq:stochastic_continuous} has a unique solution in time interval $(0, T)$ for some $T>0$ in $S_{\delta}$ for any positive $\delta > 0$.
Note that
$$ S = \ds\bigcup_{\delta > 0} S_{\delta}.$$
It then follows that \eqref{eq:stochastic_continuous} has a unique solution in $S$.

By the existence and uniqueness theorem on systems of ODE, to show \eqref{eq:stochastic_continuous} has a unique solution in $S_{\delta}$, it suffices to show that $\kappa_e w_e$ is (uniformly) Lipschitz on $S_{\delta}$.

Let $D$ be the metric on $S_{\delta}$ induced by the $\infty$-norm, i.e., given $\vec{w}, \vec{w}' \in S_{\delta}$ with $\vec{w} = \langle w_1, \ldots, w_m \rangle$ and $\vec{w}' = \langle w_1', \ldots, w_m' \rangle$, $D(\vec{w},\vec{w}') = \max_{i\in [m]} |w_i - w_i'|$. We now show that for a given edge $e$, the function $\mu_e: \vec{w} \to \kappa_e(\vec{w}) w_e$ is Lipschitz continuous on $S_{\delta}$ equipped with the metric $D$.

Fix $e = xy$. Let $\vec{w},\vec{w}' \in S_{\delta}$ be arbitrarily chosen. By Theorem \ref{thm:curvature_laplacian},
$$\kappa(x,y) = \ds\inf_{\substack{f \in Lip(1)\\\nabla_{yx}f=1}} \nabla_{xy} \Delta f.$$

Note that $|w'_{h} - w_{h}| < \epsilon$ for any edge $h$ by our assumption. WLOG that $w'_{xy} > w_{xy}$ and write $w'_{xy} - w_{xy} =\epsilon_0\leq \epsilon$. Let $f$ be the function that achieves $\inf \{(\Delta f(x) - \Delta f(y)): f \in Lip(1), f(y)-f(x)=w_{xy}\}$ under $\vec{w}$. Note for these $f$, $f(y)-f(x)=w_{xy}< w'_{xy}$. According to Lemma \ref{lem:lip-extension}, there exists a 1-Lipschitz function $f^0$ under $\vec{w'}$ such that $|f^0(z)-f(z)|\leq \epsilon$. 
It follows from Lemma \ref{lem:near-lip} that there exists $f' \in Lip(1)$ under $\vec{w'}$ such that $f'(y)- f'(x) = w'_{xy}$ and $|f^0(z)-f'(z)| \leq \epsilon_0$. Thus, $$|f(z)-f'(z)| \leq  \epsilon+\epsilon_0.$$ 
Therefore, 
$$\kappa_e' w_e'=\ds\inf_{\substack{f \in Lip(1)\\f(y)-f(x)=w_{xy}'}}(\Delta f(x)-\Delta f(y)) \leq \Delta f'(x)-\Delta f'(y).$$

It follows that if $\kappa_e' w_e' \geq \kappa_e w_e$, 
\begin{align}\label{equ:casea}
\begin{split}
    |\mu_e(\vec{w}')-\mu_e(\vec{w})| & = \kappa_e' w_e' - \kappa_e w_e \\
    &=\ds\inf_{\substack{f \in Lip(1)\\f(y)-f(x)=w_{xy}'}}(\Delta f(x)-\Delta f(y)) -  \Big(\Delta f(x)-\Delta f(y)\Big) \\
    & \leq (\Delta f'(x)-\Delta f'(y)) - (\Delta f(x) - \Delta f(y))\\
    & \leq |\Delta f'(y)-\Delta f(y)| + |\Delta f'(x)-\Delta f(x)|.
\end{split}
\end{align}

While if $\kappa_e' w_e' \leq \kappa_e w_e$,  let $g'$ be the function that achieves $$\inf \{(\Delta g'(x) - \Delta g'(y)):g' \in Lip(1), g'(y)-g'(x)=w'_{xy}\}.$$ Note for these $g'$, $g'(y)-g'(x)=w'_{xy}> w_{xy}$.  By Lemma \ref{lem:near-lip2}, it follows that there exists $g\in Lip(1)$ such that $g(y)- g(x) = w_{xy}$ and $|g(z)-g'(z)| \leq \epsilon_0$. 
Then \begin{align}\label{equ:caseb}
\begin{split}
    |\mu_e(\vec{w}')-\mu_e(\vec{w})| & = \kappa_e w_e- \kappa_e' w_e' \\
    &=\ds\inf_{\substack{f \in Lip(1)\\f(y)-f(x)=w_{xy}}}(\Delta f(x)-\Delta f(y))-\Big(\Delta g'(x)-\Delta g'(y)\Big)    \\
    & \leq (\Delta g(x) - \Delta g(y)) -  (\Delta g'(x)-\Delta g'(y)) \\
    & \leq |\Delta g'(y)-\Delta g(y)| + |\Delta g(x)-\Delta g'(x)|.
\end{split}
\end{align}

Next, we evaluate the right side of inequality (\ref{equ:casea}), the result for (\ref{equ:caseb}) is similar, and we omit the details.

As $|f(z)-f'(z)| \leq  \epsilon +\epsilon_0$ for all $z\in V(G)$, we have $f'(z)-f'(y)\leq f(z)-f(y) +2(\epsilon +\epsilon_0)$ and $f'(z)-f'(y)\geq f(z)-f(y) -2(\epsilon +\epsilon_0)$.

If $\Delta f'(y)-\Delta f(y)\geq 0$,
then
   \begin{align*}
|\Delta f'(y)-\Delta f(y)| & =  \dss_{z \in N(y)}\frac{\gamma(w'_{yz})}{\dss_{u\in N(y)} \gamma(w'_{yu})} (f'(z) - f'(y)) -  \dss_{z \in N(y)} \frac{\gamma(w_{yz})}{\dss_{u\in N(y)} \gamma(w_{yu})}  (f(z) - f(y))\\
    				      & \leq  \dss_{z \in N(y)} \frac{\gamma(w'_{yz})}{\dss_{u\in N(y)} \gamma(w'_{yu})} (f(z) - f(y) + 2(\epsilon +\epsilon_0)) - \dss_{z \in N(y)} \frac{\gamma(w_{yz})}{\dss_{u\in N(y)} \gamma(w_{yu})} (f(z) - f(y))\\
				      & = 2(\epsilon +\epsilon_0)+ \dss_{z \in N(y)} (\frac{\gamma(w'_{yz})}{\dss_{u\in N(y)} \gamma(w'_{yu})}-  \frac{\gamma(w_{yz})}{\dss_{u\in N(y)} \gamma(w_{yu})} )(f(z) - f(y))
   \end{align*}

If $\Delta f'(y)-\Delta f(y)\leq 0$,
then
   \begin{align*}
|\Delta f'(y)-\Delta f(y)| & =   \dss_{z \in N(y)} \frac{\gamma(w_{yz})}{\dss_{u\in N(y)} \gamma(w_{yu})}  (f(z) - f(y)) -\dss_{z \in N(y)}\frac{\gamma(w'_{yz})}{\dss_{u\in N(y)} \gamma(w'_{yu})} (f'(z) - f'(y))\\
    				      & \leq 	 \dss_{z \in N(y)} \frac{\gamma(w_{yz})}{\dss_{u\in N(y)} \gamma(w_{yu})} (f(z) - f(y))-  \dss_{z \in N(y)} \frac{\gamma(w'_{yz})}{\dss_{u\in N(y)} \gamma(w'_{yu})} (f(z) - f(y) - 2(\epsilon +\epsilon_0)) \\
				      & =2(\epsilon +\epsilon_0) - \dss_{z \in N(y)} (\frac{\gamma(w'_{yz})}{\dss_{u\in N(y)} \gamma(w'_{yu})}-  \frac{\gamma(w_{yz})}{\dss_{u\in N(y)} \gamma(w_{yu})}) (f(z) - f(y))
   \end{align*}

Let $C$ be the  Lipschitz constant for the $\gamma$ function. 
 Since $\gamma$ is a positive Lipschitz continuous function over $[\delta, 1]$, then there exist $M>0$ so that $\gamma\geq M$.

For both cases, we have
  \begin{align*}
|\Delta f'(y)-\Delta f(y)| & =  \abs*{\dss_{z \in N(y)}\frac{\gamma(w'_{yz})}{\dss_{u\in N(y)} \gamma(w'_{yu})} (f'(z) - f'(y)) -  \dss_{z \in N(y)} \frac{\gamma(w_{yz})}{\dss_{u\in N(y)} \gamma(w_{yu})}  (f(z) - f(y))}\\
				      & \leq 4\epsilon + \dss_{z \in N(y)} |\frac{\gamma(w'_{yz})}{\dss_{u\in N(y)} \gamma(w'_{yu})}-  \frac{\gamma(w_{yz})}{\dss_{u\in N(y)} \gamma(w_{yu})}| |f(z) - f(y)| \\
				      &\leq 4\epsilon + \dss_{z \in N(y)}  |\frac{\gamma(w'_{yz})}{\dss_{u\in N(y)} \gamma(w'_{yu})}-  \frac{\gamma(w_{yz})}{\dss_{u\in N(y)} \gamma(w_{yu})}|  w_{yz}\\
				      &\leq 4\epsilon + \dss_{z\in N(y)} \abs*{\frac{\gamma(w'_{yz})}{\dss_{u\in N(y)} \gamma(w'_{yu})} -   \frac{\gamma(w_{yz})}{\dss_{u\in N(y)} \gamma(w_{yu})} } \\				  	    
				      &\leq 4\epsilon + \dss_{z\in N(y)}\frac{| \gamma(w'_{yz})- \gamma(w_{yz})|}{\min \lp \dss_{u\in N(y)} \gamma(w_{yu}),\dss_{u\in N(y)} \gamma(w'_{yu})\rp} \\
				       &\leq 4\epsilon +C\epsilon d_y \frac{1}{d_y M}\\
				      &\leq 4\epsilon +\frac{C\epsilon}{M}.
 \end{align*}

Similarly,  $ |\Delta f'(x)-\Delta f(x)| \leq 4\epsilon +\frac{C\epsilon}{M} $. It follows that
$$ |\mu_e(\vec{w}')-\mu_e(\vec{w})| \leq (8+\frac{2C}{M}) \epsilon.$$ This completes the proof that (\ref{eq:stochastic_continuous}) has a unique solution in time interval $(0, T)$ for some $T>0$. Now we further prove that $T$ can be extended to infinity. It is enough to prove the right-hand side of (\ref{eq:stochastic_continuous}) is linearly bounded by $w$. 

For any edge $h$, we have $0<w_h\leq 1$, by Lemma \ref{coro:curvaturebound}, we then have $-\frac{2}{w_h}< \kappa_h\leq 2$. Then we get
\begin{align*}
    \begin{split}
       \frac{d w_e}{d t} &>-2w_e+w_e \sum_{h\in E(G)} (-\frac{2}{w_h})w_h\\
        &> (-2-2|E(G)|)w_e 
    \end{split}
\end{align*}
and
\begin{align*}
    \begin{split}
       \frac{d w_e}{d t} 
       &<-(-\frac{2}{w_e})w_e + w_e \sum_{h\in E(G)} 2w_h\\
       &< 2w_e +2.
    \end{split}
\end{align*}
Since for all edges $e$, $w_e(0)e^{(-2-2|E(G)|)t}$ and $-1+(w_e(0)+1)e^{2t}$ are right-bottom and right-top solutions to the component problem  of (\ref{eq:stochastic_continuous}) and both of them exist in time interval $[0, \infty)$, then so does the solution of (\ref{eq:stochastic_continuous}). 
This completes the proof of Theorem \ref{thm:sol_continuous}.

\end{proof}

We define the unnormalized continuous Ricci flow system of equations as: 
\begin{equation}\label{eq:unnor_continuous}
\begin{cases}
    \frac{dw_e}{dt} = -\kappa_e w_e \\
      X(0) = X_0.
\end{cases}
\end{equation}
Changing (\ref{eq:stochastic_continuous}) in algorithm \ref{alg:continuousricciflow} by (\ref{eq:unnor_continuous}), we have the following corollary. 

\begin{corollary}
For any initial weighted graph $G$,  by fixing the exit condition (I), there exists a unique solution $X(t)$, for all time $t\in [0, \infty)$,  to the system of ordinary differential equations in \eqref{eq:unnor_continuous}. 
\end{corollary}

\section{Solutions to  the continuous Ricci flow}\label{sec:SolutionstothecontinuousRicciflow}
In this section we will exhibit some of the solutions of the continuous Ricci flow. On general graphs, there is no explicit expression for Ricci curvature, for each edge, $\kappa_e(t)$ can be expressed independently as a infimum of expression involved continuous functions. In addition, the Right-hand-side of (\ref{eq:stochastic_continuous}) is non-linear,  these make it not easy to study the convergence result of Ricci flow.

In order to reduce the number of metric variables $w_e$ evolved in the ODE, we solve the Ricci flow on path of length $2$.
Let $G=(V, w, \mu)$ be defined on a weighted path of length $2$. Denote the vertices in $V$ as $\{x, z, y\}$. By definitions, the function $\mu=\{\mu_x^{\alpha}, \mu_y^{\alpha}, \mu_z^{\alpha}\}$ is as follows:

\begin{align*}
\mu_x^{\alpha}(v)=
    \begin{cases}
      \alpha & \mbox{if}\   v=x\\
     1-\alpha &  \mbox{if}\   v=z,
    \end{cases}\ \ \
    \mu_y^{\alpha}(v)=
    \begin{cases}
      \alpha & \mbox{if}\   v=y\\
     1-\alpha &  \mbox{if}\   v=z,
    \end{cases}\ \ \
    \mu_z^{\alpha}(v)=
    \begin{cases}
      \alpha & \mbox{if}\  v=z\\
     a_x(1-\alpha) &  \mbox{if}\   v=x\\
     a_y(1-\alpha) &  \mbox{if}\   v=y,
    \end{cases}
\end{align*}
where $a_x=\gamma(w_{zx})/(\gamma(w_{zx})+\gamma(w_{zy}) )$ and $a_y=\gamma(w_{zy})/(\gamma(w_{zx})+\gamma(w_{zy}) )$, simply speaking,  $a_x, a_y$ are functions of $w(t)$. %

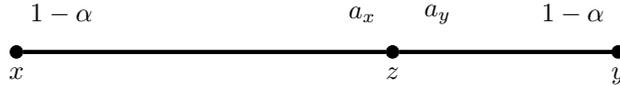
\begin{figure}[H]
 \centering
\begin{tikzpicture}[scale=2, vertex/.style={circle, draw=black, fill=black}]
\node at (0.5,0) [vertex] (v1) [label=below:{$z$}, scale=0.5] {};
\node at (-2,0) [vertex] (v2) [label=below:{$x$}, scale=0.5] {};
\node at (2,0) [vertex] (v3) [label=below:{$y$},  scale=0.5] {};
\draw [draw=black, ultra thick] (v1) -- (v2);
\draw [draw=black, ultra thick] (v1) -- (v3);
\node at (0.3,1/4) {$a_{x}$} ;
\node at (0.8,1/4) {$a_y$} ;
\node at (-1.7,1/4) {$1-\alpha$};
\node at (1.7,1/4) {$1-\alpha$};
\end{tikzpicture}
\caption{Path of length $2$.}
\end{figure}
The Ollivier-Lin-Lu-Yau curvature $\kappa$ is then as follows:
$$\kappa_{xz}=1 + a_x - a_y\frac{w_{yz}}{w_{xz}}, \
 \kappa_{yz} =1 + a_y - a_x\frac{w_{xz}}{w_{yz}}. $$
By \eqref{eq:stochastic_continuous}, we have that
$$\frac{d w_{xz}}{d t}=w_{yz} - a_x, \ \frac{d w_{yz}}{d t}=w_{xz}-a_y.$$

\subsection{Unnormalized continuous Ricci flow}
First we give an example showing that the unnormalized Ricci flow (\ref{eq:unnor_continuous}) would converge to a point if the initial metric satisfies a certain conditions. On the path graph of length $2$, for arbitrary choice of $\gamma$, we have a system of homogeneous linear differential equations:
\begin{align}\label{example:unormalized}
\begin{cases}
    \frac{d w_{xz}}{d t} &= -(1+a_x)w_{xz} +a_yw_{yz}, \\
    \frac{d w_{yz}}{d t} &= a_xw_{xz}  -(1+a_y) w_{yz}.
\end{cases}
\end{align}
Since $a_x+a_y=1$, then the associated matrix always has eigenvalues $\lambda_1=-1, \lambda_2=-2$. If we set $a_x=a_y=\frac{1}{2}$, then corresponding eigenvectors are $(0.7071, 0.7071)^T$, and $(0.7071, -0.7071)^T$, then (\ref{example:unormalized}) has solution of form:
\begin{align*}
    \begin{cases}
     w_{xz}(t)=0.7071(c_1e^{-t} + c_2e^{-2t}),\\
     w_{yz}(t)=0.7071(c_1e^{-t} - c_2e^{-2t}).
    \end{cases}
\end{align*}
If the initial metric satisfies $w_{xz}(0)=w_{yz}(0)$, i.e. $c_2=0$, then $w_{xz}(t)=w_{yz}(t)=0.7071\times c_1e^{-t}$. Thus both $w_{xz}(t),  w_{yz}(t)$ are decreasing functions with time $t$ which implies that
the edge length converge to zero, in this case the graph converges to a point.

\subsection{Normalized continuous Ricci flow}\label{sec:examples} 
Although by Theorem \ref{thm:sol_continuous} we are guaranteed to have an unique solution to the system of ODEs in \eqref{eq:stochastic_continuous}, the types of solutions we obtain depend on the choice of $\gamma$ in \eqref{eq:weight_dist} and sometimes the initial condition. In this subsection, we give examples of different solutions to \eqref{eq:stochastic_continuous} defined on the same path graph of length $2$. 
The results also answers the question asked at the beginning of the paper, we will see that the limit of the Ricci-flow on path exists, and it is possible to have a constant curvature although the initial graph does not have.

\begin{example}

\begin{description}

\item[\textbf{Constant solution:}] If we pick $a_x = w_{yz}$ and $a_y = w_{xz}$,
note $\gamma$ is the function satisfying $\gamma(w_{zx})/\gamma(w_{zy}) =w_{zy}/w_{zx}$, then $|\gamma(w_{zx})-\gamma(w_{zy})|=C|w_{zy}-w_{zx}|$ for some constant $C$.
It follows that $\frac{d w_{xz}}{dt}=\frac{dw_{yz}}{dt}=0$. Hence $w_{xz}(t) = w_{xz}(0)$ and $w_{yz}(t) = w_{yz}(0)$ for all $t$ and

$$\kappa_{xz}(t)=\kappa_{yz}(t)= 1. $$

\item[\textbf{Stable solution without collapsing:}] If we pick $a_x = w_{xz}$ and $a_y = w_{yz}$, note  $\gamma$ is the function satisfying $\gamma(w_{zx})/\gamma(w_{zy}) =w_{zx}/w_{zy}$, then $|\gamma(w_{zx})-\gamma(w_{zy})|=C|w_{zx}-w_{zy}|$ for some constant $C$.
Then
\begin{align*}
    \frac{d w_{xz}}{d t} &= w_{yz} - w_{xz} = 1 - 2w_{xz},\\
    \frac{d w_{yz}}{d t} &= w_{xz} - w_{yz} = 1 - 2w_{yz}.
\end{align*}
It follows that
\begin{align*}
    w_{xz}(t) &=\frac{1}{2} - \lp \frac{1}{2}-w_{xz}(0)\rp e^{-2t}, \\
    w_{yz}(t) &=\frac{1}{2} - \lp \frac{1}{2}-w_{yz}(0)\rp e^{-2t}.
\end{align*}
Thus $w_{xz}(t) \to \frac{1}{2}$ and $w_{yz}(t) \to \frac{1}{2}$ as $t \to \infty$,
and
$$\kappa_{xz}(t)\to 1, ~\kappa_{yz}(t)\to 1. $$

\item[\textbf{Stable solution with collapsing:}] Suppose WLOG that $w_{xz}(0) > w_{yz}(0)$. If we pick $a_x = \frac{w_{yz}^2}{w_{xz}^2+w_{yz}^2}$ and $a_y = \frac{w_{xz}^2}{w_{xz}^2+w_{yz}^2}$, note $\gamma$ is the function satisfying $\gamma(w_{zx})/\gamma(w_{zy}) =w_{zy}^2/w_{zx}^2$, then $|\gamma(w_{zx})-\gamma(w_{zy})|=C|w_{zy}^2-w_{zx}^2|=C|w_{zy}-w_{zx}|$ for some constant $C$.

Then
\begin{align*}
    \frac{d w_{xz}}{d t} &= w_{yz} - \frac{w_{yz}^2}{w_{xz}^2+w_{yz}^2} = \frac{w_{yz}(w_{xz}^2 - w_{yz} w_{xz})}{w_{xz}^2+w_{yz}^2} > 0, \\
    \frac{dw_{yz}}{dt} &= w_{xz} - \frac{w_{xz}^2}{w_{xz}^2+w_{yz}^2} = \frac{w_{xz}(w_{yz}^2 - w_{yz} w_{xz})}{w_{xz}^2+w_{yz}^2} < 0.
\end{align*}
It follows that $w_{xz}(t) \to 1$ and $w_{yz}(t) \to 0$ as $t \to \infty$
and
$$\kappa_{xz}(t)\to 2.$$ The edge $yz$ converges to point $z$ eventually.

\end{description}

\end{example}

\section{Convergence of Ricci flow}\label{sec:ConvergenceofRicciflow}

In this section, we prove  convergence result of Ricci flow on  path and star graphs equipped with any initial weight. 
A graph minor is obtained from a given graph by repeatedly removing or contracting edges. From the path instance,  we will also see its graph minors under Ricci flow accompanied with appropriated edge operations.  

\subsection{Ricci flow on path}
Let P be a finite path of length $n\geq 3$,  denote the edge set of P as $\{e_{i}\}_{i= 1}^{n}$ where  $e_1, e_{n}$ are leave edges. 
We prove the following result:
\begin{theorem}\label{thm:pathconverge}
Let $\gamma(x)=\frac{1}{x}$. Ricci flow (\ref{eq:stochastic_continuous}) on path P  converges. By contracting edges with small weights, any initial weighted path will converge to a path of length $2$. 
\end{theorem}

\begin{proof}
Recall $D_u=\sum_{x\in N(u)}\gamma(w_{ux})$, by calculation, if $e=(u, v)$ is non-leaf edge with $x-u-v-y$,  
the Ollivier-Lin-Lu-Yau Ricci curvature is
\begin{align*}
    \kappa_{uv}&=\frac{\gamma(w_{uv})}{D_u}+\frac{\gamma(w_{vu})}{D_v}-\frac{w_{ux}}{w_{uv}}\frac{\gamma(w_{ux})}{D_u}-\frac{w_{vy}}{w_{uv}} \frac{\gamma(w_{vy})}{D_v}\\
    &=\frac{\frac{1}{w_{uv}}}{\frac{1}{w_{uv}}+\frac{1}{w_{ux}}} +\frac{\frac{1}{w_{uv}}}{\frac{1}{w_{uv}}+\frac{1}{w_{vy}}} -\frac{1}{w_{uv}(\frac{1}{w_{uv}}+\frac{1}{w_{ux}})}-\frac{1}{w_{vu}(\frac{1}{w_{uv}}+\frac{1}{w_{vy}})}\\
    &=0.
\end{align*}
If $e=(x, u)$ is a leaf edge with with $x-u-v$ and $d_x=1$, then
 \begin{align*}
    \kappa_{ux}&=1+\frac{\gamma(w_{ux})}{D_u}-\frac{w_{uv}}{w_{ux}}\frac{\gamma(w_{uv})}{D_u}\\
    &=1+\frac{\frac{1}{w_{ux}}}{\frac{1}{w_{uv}}+\frac{1}{w_{ux}}} -\frac{1}{w_{ux}(\frac{1}{w_{uv}}+\frac{1}{w_{ux}})}\\
    &=1. 
\end{align*}
Then  $\sum_{h\in E} \kappa_hw_h =w_{e_1}+w_{e_{n}}, $ since $\sum_e w_e(t)= 1$ for all time $t$, then 
$$\frac{d w_{e_i}}{d t} =w_{e_i}(-1+w_{e_1}+w_{e_{n}})<0, ~~ i\in \{1, n\},$$ 
$$\frac{d w_{e_i}}{d t} =w_{e_i}(0+w_{e_1}+w_{e_{n}})> 0, ~~ i\in \{2,\ldots, n-1\}.$$ 

Then $w_{e_1}(t),  w_{e_n}(t)$ decrease  monotonically and weight $w_{e_i}$ on non-leave edges $e_i$ increase  monotonically, 
by repeatedly contracting edges with small enough weight (leaves), eventually the path converge to a path of length two.  
Refer to the constant solution of Example 1,  weights on these two edges will not change any more.

\end{proof}

\subsection{Ricci flow on star}
A star is a tree with one internal node and a number of leaves. 
By choosing $\gamma(x)=\frac{1}{x}$, we can prove that Ricci flow on any initial weighted star graph converges.

\begin{figure}[ht]
    \centering
 \includegraphics[scale=0.3]{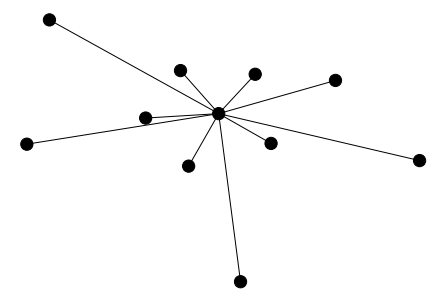}
 \hfil
 \includegraphics[scale=0.3]{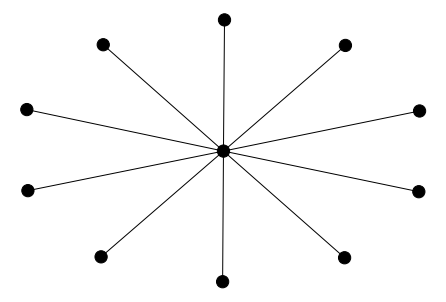}
 \caption{An non-constant weighted star converges to constant-weighted star}
\end{figure}

\begin{theorem}\label{thm:starconverge}
Let $G$ be a star with initial weight $\vec{w}(0)$ and $\gamma(x)=1/x$. 
Denote the internal node of $S_n$ as $u$, leave as $x$, $3\leq d_u<\infty$. 
The Ricci flow  (\ref{eq:stochastic_continuous}) deforms the weight $\vec{w}(0)$ of $G$ to a  weight of constant value $\frac{1}{d_u}$. 
\end{theorem}

To prove Theorem \ref{thm:starconverge}, we first need to prove  $w_e(t)$ is a monotone function for any edge $e$ in the graph, thus the $\lim_{t\to \infty}w_e(t)$ exists and is finite, then we prove the limit of $w_e(t)$ is $\frac{1}{d_u}$.
Let $F_e(t)=\sum \kappa_h(t) w_h(t) - \kappa_e(t)$. Then $\frac{d w_e}{d t}=w_e(t)F_e(t)$. 
We will prove the following result.

\begin{lemma}\label{thm:alwaygreaterthan0}
Let $G$ be a star and $\gamma(x)=1/x$. 
For any edge $e\in E(G)$, if $F_e(0)\geq 0$, then for all $t\in (0, \infty)$, $F_e(t)\geq 0$;  if $F_e(0)\leq 0$, then for all $t\in (0, \infty)$, $F_e(t)\leq 0$. Thus $w_e(t)$ is a monotone function over time $t$. 
\end{lemma}
To prove above lemma, we need a result $|\frac{d F_e(t)}{d t}|\leq 
 C|F_e(t)|$ for all $t$,  where $C$ is  finite. This can be seen from the following  facts. 

\begin{lemma}[Hamilton's lemma]\label{lem:hamiltonlemma}\cite{hami1986}
Let $f$ be a locally Lipschitz continuous function on $[a, b]$.  (1): $f(a)\leq 0$, and when  $f\geq 0$  we have $df/dt \leq 0$, then $f(b)\leq 0$. 

Conversely, (2): let $f(a)\geq 0$, and when $f\leq 0$  we have $df/dt \geq 0$, then $f(b)\geq 0$.
\end{lemma}

\begin{proof}
We prove item (1). 
By contradiction, assume $f(b)>0$. We need an auxiliary function $g=e^{-x}f$ defined on $[a, b]$.   Since $f(b)>0$, then $g(b)>0$.  Consider the maximal point $t_0\in [a, b]$ of $g$, then $g(t_0)\geq g(b)>0$, thus $f(t_0)>0$. We have $dg/dt=-e^{-x}f+e^{-x}df/dt=e^{-x}(df/dt-f)$, and $\frac{dg}{dt}(t_0)$ is strictly less than $0$ as $f(t_0)$ is strictly greater than $0$ and $\frac{df}{dt}(t_0)\leq 0$. 

Note $g$ is a locally Lipschitz continuous function, then $\frac{dg}{dt}(t_0)<0$ means $\lim_{h\to 0} sup \frac{g(t_0+h)-g(t_0)}{h}<0$. Since $t_0\leq b$, thus for any $h<0$ such that $g(t_0+h)-g(t_0)>0$,  which is a contradiction to the maximal point $t_0$ of $g$.  

Proof of item (2) is similar. 
\end{proof}

The following uses Lemma \ref{lem:hamiltonlemma} directly. 
\begin{corollary}[Hamilton's Corollary]\label{coro:hamiltoncoro}
 Let $f$ be a locally Lipschitz continuous function on $[a, b]$. Let $c$ represent a finite positive value.  If  $|df/dt|\leq c|f|$,  then $f(b)\leq 0$ if $f(a)\leq 0$;   $f(b)\geq 0$ if $f(a)\geq 0$. 
\end{corollary}
\begin{proof}
We prove the first result.  Let $g=e^{-ct}f$, then $g(a)\leq 0$ as $f(a)\leq 0$ by condition. If for all $a\leq t<b$, $f(t)\leq 0$, then $f(b)\leq 0$ by continuity of $f$. Assume there exist $t<b$ such that $f(t)>0$, then  we have $\frac{df}{dt} \leq cf$. 
Further we have $\frac{dg}{dt}=-ce^{-ct}f + e^{-ct}\frac{df}{dt}  \leq 0$. By Lemma \ref{lem:hamiltonlemma}, $g(b)\leq 0$, thus $f(b)\leq 0$. 
\end{proof}

We need the following lemma in \cite{HM} to obtain a bound for expression $\frac{D_u^{\prime}(t)}{D_u(t)}$.  Its proof can be found in \cite{RV}. 
\begin{lemma}\label{lem:inequqqqqq}
 If q $_{1}, q_{2}, \ldots, q_{n}$ are positive numbers, then
$$
\min _{1 \leq i \leq n} \frac{p_{i}}{q_{i}} \leq \frac{p_{1}+p_{2}+\cdots+p_{n}}{q_{1}+q_{2}+\cdots+q_{n}} \leq \max _{1 \leq i \leq n} \frac{p_{i}}{q_{i}}
$$
for any real numbers $p_{1}, p_{2}, \ldots, p_{n} .$ Equality holds on either side if and only if all the ratios $p_{i} / q_{i}$ are equal.
\end{lemma}

\begin{proof}[Proof of Lemma \ref{thm:alwaygreaterthan0}]
 By calculation, for any edge $e=ux$ we have  $\kappa_{ux}=1+\frac{2-d_u}{w_{ux}D_u}$
and $F_{ux}=\frac{d_u-2}{D_u}(\frac{1}{w_{ux}}-d_u).$
Since  $\frac{1}{w_{ux}D_u}<1, \frac{1}{D_u}<\frac{1}{d_u}$, then $F_{ux}$ is bounded at any finite time. 

The derivative of $F_{ux}(t)$ respect to $t$ is
\begin{align*}
    F_{ux}^{\prime}&= -\frac{(d_u-2)w_{ux}^{\prime}}{w_{ux}^2D_u}-\frac{D_u^{\prime}}{D_u}\frac{d_u-2}{D_u}(\frac{1}{w_{ux}}-d_u)\\
    &=-\frac{(d_u-2)w_{ux}F_{ux}}{w_{ux}^2D_u}-\frac{D_u^{\prime}}{D_u}F_{ux}\\
    &=(-\frac{(d_u-2)}{w_{ux}D_u}-\frac{D_u^{\prime}}{D_u})F_{ux}. 
\end{align*}

Using Lemma \ref{lem:inequqqqqq}, we get 
\begin{align*}
    |\frac{D_u^{\prime}}{D_u}| &=|\frac{\sum_{z\sim u}- \frac{w_{uz}^{\prime}}{w_{uz}^2} }{\sum_{z\sim u}\frac{1}{w_{uz}}}|=|\frac{\sum_{z\sim u} \frac{F_{uz}}{w_{uz}} }{\sum_{z\sim u}\frac{1}{w_{uz}}}|\leq \frac{\sum_{z\sim u} \frac{|F_{uz}|}{w_{uz}} }{\sum_{z\sim u}\frac{1}{w_{uz}}}\leq \max_{z\sim u}{ |F_{uz}|}.
\end{align*}

Thus 
\begin{align}\label{equ:FeleqCFe}
     |F_{ux}^{\prime}(t)|\leq (d_u-2+ \max_{z\sim u}{ |F_{uz}|}) |F_{ux}|. 
\end{align}

Let $C=(d_u-2+ \max_{z\sim u}{ |F_{uz}|})$, clearly, it is a finite number.
Since $F_e(t)$ is differentiable and $F_e^{\prime}(t)$ is bounded, by Hamilton's Corollary \ref{coro:hamiltoncoro},  if $F_{e}(0)\geq 0$, then $F_{e}(t)\geq 0$ for all $t>0$;  if $F_{e}(0)\leq 0$, then $F_{e}(t)\leq 0$ for all $t>0$.   Thus $w_e(t)$ is a  monotone function over time $t\in[0, \infty)$.

\end{proof}

\begin{lemma}
[Barbalat's Lemma]\cite{barlem}\label{lem:barlem}
If $f(t)$ has a finite limit as $t\to \infty$ and if $f^{\prime}(t)$ is uniformly continuous (or $f^{\prime\prime}(t)$ is bounded), then $\lim_{t\to\infty} f^{\prime}(t)=0$.  
\end{lemma}

\begin{proof}[Proof of Theorem \ref{thm:starconverge}]
By Lemma \ref{thm:alwaygreaterthan0},  we have  $\lim_{t\to \infty}w_e (t)$ exists and is finite, both $w_e(t)$ and $F_e(t)$ are uniformly continuous, thus $w_e^{\prime}(t)=w_eF_e$ is uniformly continuous.  By Barbalat's Lemma \ref{lem:barlem},
 $\lim_{t\to \infty}w_e^{\prime}(t)=0$. 
 
 Note 
 $\lim_{t\to \infty}w_e(t)= 0$ if and only if $F_e(t)$ is negative for large $t$, equivalently, $w_e(t)>\frac{1}{d_u}$ a contradiction. 
 
 Thus, it must be $\lim_{t\to \infty} F_e(t)=0$, implies $\lim_{t\to \infty} w_e(t)=\frac{1}{d_u}$. Therefore, the Ricci flow (\ref{eq:stochastic_continuous}) on star graph $S_n$ with $n\geq 3$ converges to constant-weighted star of same size. 
\end{proof}

\section{Conclusions}
In this study, we propose an normalized continuous Ricci flow for weighted graphs,  based on Ollivier-Lin-Lu-Yau Ricci curvature and prove that the Ricci flow metric $X(t)$ with initial data $X(0)$ exists and is unique for all $t\geq 0$ by fixing the violation of distance condition. 
We also show some explicit, rigorous examples of Ricci flows on tree graphs. 
Future work already underway, we expect results of more general Ricci flows evolved on various graphs.

\bibliographystyle{plain}
\bibliography{citations}

\begin{thebibliography}{10}

\bibitem{blhy}
S.~Bai, A.~Huang, L.~Lu, and S.T. Yau.
\newblock On the sum of ricci-curvatures for weighted graphs.
\newblock {\em Pure Appl. Math. Q.}, 17(5):1599--1617, 2021.

\bibitem{barlem}
I.~Barb\u{a}lat.
\newblock Systèmes d'équations différentielles d'oscillations non
  linéaires.
\newblock {\em Rev. Math. Pures Appl}, 4:267--270, 1959.

\bibitem{BCLMP}
D.~Bourne, D.~Cushing, S.~Liu, F.~Münch, and N.~Peyerimhoff.
\newblock Ollivier--ricci idleness functions of graphs.
\newblock {\em SIAM J. Discrete Math}, 32, 04 2017.

\bibitem{spheretheo}
S.~Brendle and R.~Schoen.
\newblock Manifolds with $^{1/4}$-pinched curvature are space forms.
\newblock {\em J. Amer. Math. Soc.}, 22:287--307, 2009.

\bibitem{bookODE}
R.~Fang, D.and~Xue.
\newblock {\em Ordinary Differential Equation}.
\newblock Higher Education Press, 2017.

\bibitem{hami1982}
R.~Hamilton.
\newblock Three-manifolds with positive ricci curvature.
\newblock {\em J. Differ. Geom}, 17:255--362, 06 1982.

\bibitem{hami1986}
R.~Hamilton.
\newblock Four-manifolds with positive curvature operator.
\newblock {\em J. Differ. Geom}, 24(2):153--179, 1986.

\bibitem{LLY}
Y.~Lin, L.~Lu, and S.T. Yau.
\newblock Ricci curvature of graphs.
\newblock {\em Tohoku Math. J.}, 63, 12 2011.

\bibitem{HM}
H.~Minc.
\newblock Nonnegative matrices.
\newblock {\em Wiley, New York}, 1988.

\bibitem{MW}
Florentin Münch and Radoslaw K.~Wojciechowski.
\newblock Ollivier ricci curvature for general graph laplacians: Heat equation,
  laplacian comparison, non-explosion and diameter bounds.
\newblock {\em Adv. Math.}, 356, 11 2019.

\bibitem{NLLG}
Chien-Chun Ni, Yu-Yao Lin, Feng Luo, and Jie Gao.
\newblock Community detection on networks with ricci flow.
\newblock {\em Sci. Rep}, 9, 07 2019.

\bibitem{Ollivier}
Y.~Ollivier.
\newblock Ricci curvature of markov chains on metric spaces.
\newblock {\em J. Funct. Anal.}, 256:810--864, 02 2009.

\bibitem{Ollivier1}
Yann Ollivier.
\newblock Ricci curvature of metric spaces.
\newblock {\em C. R. Math.}, 345(11):643 -- 646, 2007.

\bibitem{perelman2002entropy}
G.~Perelman.
\newblock The entropy formula for the ricci flow and its geometric
  applications.
\newblock {\em arXiv preprint math/0211159}, 2002.

\bibitem{RV}
R.~Varga.
\newblock Matrix iterative analysis.
\newblock {\em Prentice-Hall, Englewood Cliffs, NJ}, 1962.

\end{thebibliography}

\vspace{1cm}

Yanqi Lake Beijing Institute of Mathematical Sciences and Applications, Beijing, 101408, China. 

\textit{Email address}: \textbf{sbai@seu.edu.cn}
\vspace{0.3cm}

Yau Mathematical Science Center, Tsinghua University, Beijing, 100084, China; Department of Mathematics,  Tsinghua University, Beijing, 100084, China.

\textit{Email address}: \textbf{yonglin@tsinghua.edu.cn} 
\vspace{0.3cm}

University of South Carolina, Columbia, SC 29208, USA.

\textit{Email address}: \textbf{lu@math.sc.edu}

\vspace{0.3cm}

Georgia Institute of Technology, Atlanta, GA, 30332, USA.

\textit{Email address}: \textbf{zwang672@gatech.edu}

\vspace{0.3cm}

Yau Mathematical Sciences Center, Tsinghua University, Beijing, 100084, China; Yanqi Lake Beijing Institute of Mathematical Sciences and Applications, Beijing, 101408, China.

\textit{Email address}: \textbf{styau@tsinghua.edu.cn, yau@math.harvard.edu}

\end{document}